\documentclass[11pt,a4j]{article}
\usepackage{amsthm}

    \newtheorem{thm}{Theorem}[section]
    \newtheorem{prop}[thm]{Proposition}
    \newtheorem{lem}[thm]{Lemma}

  \theoremstyle{definition}
    
    \newtheorem{defi}[thm]{Definition}
  \theoremstyle{remark}
    \newtheorem{rem}[thm]{Remark}
    \newtheorem{ex}[thm]{Example}

\usepackage{amsmath}
\usepackage{amssymb}
\usepackage[all]{xy}
\usepackage{graphicx}
\usepackage{mathrsfs}
\usepackage{indentfirst}

\title{Sheaf theoretic characterization of \'{e}tale groupoids}
\author{Koji Yamazaki}
\begin{document}
\maketitle
\begin{abstract}
The study of Haeflier \cite{haefliger1958,haefliger2006lecture} suggests that it is natural to regard a pseudogroup as an \'{e}tale groupoid.
We show that any \'{e}tale groupoid corresponds to a {\it pseudogroup sheaf}, a new generalization of a pseudogroup.
This correspondence is an analog of the equivalence of the two definitions of a sheaf: 
as an \'{e}tale space and as a contravariant functor.
\end{abstract}
\setcounter{section}{-1}
	\section{Introduction}
Let $X$ be a topological space, and let $X_{top}$ be the set of the open sets in $X$.
$X_{top}$ is an oreded set, therefore, $X_{top}$ is a small category.
A {\it presheaf} on $X$ is a functor $X_{top}^{op} \rightarrow {\bf Set}$, where {\bf Set} is the category of the (small) sets and the maps.
A presheaf $\mathcal{F}$ is a {\it sheaf} if, for any open set $U \subset X$ and any open covering $\{ U_\lambda \}$ of $U$, the following diagram is an equalizer:
\[
\mathcal{F}(U) \rightarrow \prod_\lambda \mathcal{F}(U_\lambda) \rightrightarrows \prod_{\lambda, \mu} \mathcal{F}(U_{\lambda \mu}),
\]
where $U_{\lambda \mu} = U_\lambda \cap U_\mu$.
The category of the sheaves on $X$ is equivalent to the category of the \'{e}tale spaces over $X$.
An {\it \'{e}tale space} over $X$ is a local homeomorphism $E \rightarrow X$ from a topological space $E$ to the given topological space $X$.
(cf. \cite{forster2012lectures} or Section 1.) \par
For two open sets $U$ and $V$ in $X$, let $\operatorname{Homeo}_X (U,V)$ be the set of the homeomorphisms from $U$ to $V$.
The correction of the sets $\operatorname{Homeo}_X (U,V)$ determines a groupoid $\operatorname{Homeo}_X$.
A {\it pseudogroup} is a subgroupoid $\mathcal{H}$ of $\operatorname{Homeo}_X$ satisfying the sheaf property:
\begin{itemize}
  \item For any homeomorphism $f \in \operatorname{Homeo}_X (U,V)$, $f$ belongs to $\mathcal{H}(U,V)$ if and only if there exists an open covering $\{ U_\alpha \}$ of $U$ such that the restrictions $f |_{U_\alpha}$ belong to $\mathcal{H} (U_\alpha , f(U_\alpha))$.
\end{itemize}
According to Haeflier \cite{haefliger1958,haefliger2006lecture}, it is natural to regard pseudogroups as \'{e}tale groupoids.
A {\it topological groupoid} (over $X$) is a groupoid object in the category of the topological spaces (such that the object space is $X$). 
A topological groupoid is {\it \'{e}tale} if the source map is a local homeomorphism
(cf. \cite{moerdijk2003introduction} or Section 1.1).
Resende \cite{resende2007etale} introduced an {\it abstract pseuedogroup} as a complete and infinitely distributive inverse semigroup, and he gave equivalent correspondence among the \'{e}tale groupoids, the abstract pseuedogroups and the {\it quantales}, by locale theory. \par
We define a {\it pseudogroup sheaf} (cf. Definition \ref{pseudogroupsheaf}), a new generalization of a pseudogroup, and prove the equivalence of the \'{e}tale groupoids over $X$ and the pseudogroup sheaves on $X$, as an analogue of the equivalence of the \'{e}tale spaces over $X$ and the sheaves on $X$.
The main result in this paper is the following theorem.
\setcounter{section}{4}
\setcounter{thm}{2}
\begin{thm}
Let $X$ be a $T_1$ space. 
Then, there exists an equivalence ${\bf Pse}(X) \simeq \mbox{\bf \'{E}tGrpd}_X$.
\end{thm}
{\bf Pse}$(X)$ is the category of the pseudogroup sheaves on $X$, and {\bf \'{E}tGrpd}$_X$ is the category of the \'{e}tale groupoids over $X$.
The requirement that the base space is a $T_1$ space in Theorem 4.3 is contrasted with the requirement that the base space is a sober space in Resende's result. \par
\vspace{5pt}
We review sheaves and \'{e}tale spaces in Section 1.
We see the correspondence from \'{e}tale groupoids to pseudogroup sheaves in Section 2.
We define a pseudogroup sheaf, and see the correspondence from pseudogroup sheaves to \'{e}tale groupoids in Section 3.
The main result is obtained by summarizing the contents of the previous sections in categorical viewpoints in Section 4.
As an added note, the interpretation of a pseudogroup in classical sence as a pseudogroup sheaf is given in Section 5. \par
In this paper, we write the operation of a given category in three different manners.
\begin{quote}
$(-) \circ (-)$ is the composition of two maps. \\
$(-) \cdot (-)$ is the operation of a given \'{e}tale groupoid. \\
$(-) \bullet (-)$ is the operation of a given pre-pseudogroupoid.
\end{quote}
We often call a morphism ``an arrow'' in a given \'{e}tale groupoid or a given pre-pseudogroupoid.
The purpose of these manners is to eliminate confusion.
\setcounter{section}{0}
\setcounter{thm}{0}
	\section{Sheaves}
Let $X$ be a topological space.
A {\it presheaf} on $X$ is a contravariant functor on the ordered set of the open sets in $X$ with valued in the category of the (small) sets and the maps.
A {\it sheaf} on $X$ is a presheaf satisfying the sheaf condition (cf. Definition \ref{sheaf}).
The category of the sheaves on $X$ is equivalent to the category of the \'{e}tale spaces over $X$ (cf. Section1.5).
	\subsection{Preliminaries}
Let $X$ be a topological space.
First, we mention some notations for topological spaces in this paper.
\begin{defi} \label{top}
\mbox{}
\begin{itemize}
\item $X_{top}$ is the set of the open sets in $X$.
\item $\mathcal{N} (x)$ is the set of the open neighborhoods of a point $x \in X$.
\item For two points $x, y \in X$, we denote $x \rightarrow y$ if $\mathcal{N} (y) \subset \mathcal{N} (x)$.\\
(In other words, we denote $x \rightarrow y$ if a sequence $\{ x \}$ converges to $y$.)
\item A topological space $X$ is a {\it $T_1$ space} if a relation $x \rightarrow y$ implies the equality $x=y$ for any two points $x, y \in X$.
\end{itemize}
\end{defi}
$X_{top}$ is an oreded set, therefore, $X_{top}$ is a small category.
We also mention some notations for categories.
\begin{defi} \label{cat}
\mbox{}
\begin{itemize}
\item $Ob(\mathcal{C})$ is the class of the objects of a given category $\mathcal{C}$.
\item $\mathcal{C}^{op}$ is the opposite category of a category $\mathcal{C}$.
\item A set is {\it small} if it is in the given universe (cf. \cite{mac2013categories}). \\
A category is {\it small} if the set of the objects and the set of the morphisms are small.
\item {\bf Set} is the category of the (small) sets and the maps.
\item A {\it filtered set} $\Lambda$ is an ordered set $\Lambda$ such that, for any two element $\lambda, \mu \in \Lambda$, there exists an element $\nu \in \Lambda$ such that it satisfies the relations $\lambda, \mu \le \nu$.
\item A {\it direct system} (resp. a {\it inverse system}) indexed by a filtered set $\Lambda$ is a functor $\Lambda \rightarrow {\bf Set}$ (resp. $\Lambda^{op} \rightarrow {\bf Set}$).
\item A {\it direct limit} $\displaystyle{\lim_{\longrightarrow}} F$ (resp. a {\it inverse limit} $\displaystyle{\lim_{\longleftarrow}} G$) of a direct system $F$ (resp. a inverse system $G$) indexed by a filtered set $\Lambda$ is a set defined by the follows:
\begin{itemize}
\item[$\circ$] $\displaystyle{\lim_{\longrightarrow}} F = \coprod_\lambda F_\lambda /{\sim}$, \\
where $\sim$ is a equivalent relation defined by the follows:
\[
x \sim y \, (x \in F_\lambda, y \in F_\mu) \Leftrightarrow 
^\exists \nu \ge \lambda, \mu \mbox{ s.t. } F_{\lambda \le \nu}(x) = F_{\mu \le \nu}(y).
\]
\item[$\circ$] $\displaystyle{\lim_{\longleftarrow}} G = \{ (x_\lambda) \in \prod_\lambda G_\lambda \, | \, G_{\mu \ge \lambda}(x_\lambda) = x_\mu \}$.
\end{itemize}
\end{itemize}
\end{defi}
	\subsection{Sheaves}
We define sheaves on a topological space.
\begin{defi}[cf. \cite{forster2012lectures}] \label{sheaf}
A {\it presheaf} on a topological space $X$ is a functor $X_{top}^{op} \rightarrow {\bf Set}$.
A presheaf $\mathcal{F}$ is a {\it sheaf} (or a {\it sheaf of sections}) if, for any open set $U \subset X$ and any open covering $\{ U_\lambda \}$ of $U$, the following diagram is an equalizer:
\[
\mathcal{F}(U) \rightarrow \prod_\lambda \mathcal{F}(U_\lambda) \rightrightarrows \prod_{\lambda, \mu} \mathcal{F}(U_{\lambda \mu}),
\]
where $U_{\lambda \mu} = U_\lambda \cap U_\mu$.
A {\it section} of a (pre)sheaf $\mathcal{F}$ on an open set $U$ is an element $s \in \mathcal{F}(U)$.\par
A {\it morphism} between (pre)sheaves is a natural transformation between functors.
\end{defi}
	\subsection{\'{E}tale spaces}
The concept of \'{e}tale spaces is equivalent to sheaves.
\begin{defi}[cf. \cite{forster2012lectures}]
An {\it \'{e}tale space} (or a {\it sheaf of germs}) over a topological space $X$ is a local homeomorphism $E \rightarrow X$ from a topological space $E$ to $X$.\par
Let $p : E \rightarrow X$ and $q : F \rightarrow X$ be \'{e}tale spaces over $X$.
A {\it morphism} $f : p \rightarrow q$ over $X$ is a continuous map $f : E \rightarrow F$ such that the following diagram is commutative:
	\[\xymatrix{
		E \ar[r]^-{f} \ar[d]_-{p}
		&F \ar[d]^-{q}
	\\
		X \ar@{=}[r]
		& X
	}\]
\end{defi}
Let $p : E \rightarrow X$ be a \'{e}tale space over $X$.
The sheaf $\Gamma(-; p)$ associated with $p$ is defined by the follows:
\[
\Gamma(U; p) = \{ s : U \rightarrow E \, | \, p(s(x)) = x \, (^\forall x \in U) \}
\]
	\subsection{\'{E}tale spaces associated with presheaves}
We will consider the inverse construction of the previous subsection.
Let $\mathcal{F}$ be a presheaf on $X$.
Let $x \in X$ be a point.
The set $\mathcal{N}(x)$ of open neighborhoods of $x$ is an oreded set, therefore, $\mathcal{N}(x)$ is a small category.
The opposite category $\mathcal{N}(x)^{op}$ of $\mathcal{N}(x)$ is a filtered set.
\begin{defi}[cf. \cite{forster2012lectures}]
The {\it stalk} $\mathcal{F}_x$ of $\mathcal{F}$ at $x$ is the direct limit of the direct system $\mathcal{F} |_{\mathcal{N}(x)^{op}}$.
For a section $s \in \mathcal{F}(U)$ and a point $x \in U$, the {\it germ} $s_x$ of $s$ at $x$ is the image of $s$ for the natural map $\mathcal{F}(U) \rightarrow \mathcal{F}_x$.
\end{defi}
We define a set $E_\mathcal{F}$ and a map $p_\mathcal{F} : E_\mathcal{F} \rightarrow X$ as the follows:
\[
\begin{array}{lcl}
E_\mathcal{F} & = & \displaystyle\coprod_x \mathcal{F}_x, \\
p_\mathcal{F}(s_x) & = & x \, (s_x \in \mathcal{F}_x).
\end{array}
\]
For an open set $U \subset X$ and a section $s \in \mathcal{F}(U)$, a subset $[s, U] \subset E_\mathcal{F}$ is defined as $[s, U] = \{ s_x \, | \, x \in U \}$.
\begin{prop}[cf. \cite{forster2012lectures}]
The correction of subsets $\{ [s, U] \, | \, s, U \}$ determines a unique topology on $E_\mathcal{F}$ such that $\{ [s, U] \, | \, s, U \}$ is a basis for the topology.
Moreover, the map $p_\mathcal{F} : E_\mathcal{F} \rightarrow X$ is a local homeomorphism.
\end{prop}
$p_\mathcal{F} : E_\mathcal{F} \rightarrow X$ is the \'{e}tale space associated with the presheaf $\mathcal{F}$.
	\subsection{Categorical viewpoints}
Some concrete categories are defined as follows.
\begin{quote}
{\bf PSh}$(X)$ the category of the presheaves on $X$. \\
{\bf Sh}$(X)$ the category of the sheaves on $X$. \\
{\bf \'{E}t}$_X$ the category of the \'{e}tale spaces over $X$.
\end{quote}
The construction $p \mapsto \Gamma(-; p)$ defines a functor ${\bf \Gamma} : \mbox{\bf \'{E}t}_X \rightarrow {\bf Sh}(X) \subset {\bf PSh}(X)$.
The construction $\mathcal{F} \mapsto p_\mathcal{F}$ defines a functor ${\bf p} : {\bf PSh}(X) \rightarrow \mbox{\bf \'{E}t}_X$.
Then, there exists an adjunction ${\bf p} \dashv {\bf \Gamma} : {\bf PSh}(X) \rightarrow \mbox{\bf \'{E}t}_X$.
The unit map $\mu : Id \rightarrow {\bf \Gamma} \circ {\bf p}$ of the adjunction ${\bf p} \dashv {\bf \Gamma}$ is called the {\it sheafification}.
For a presheaf $\mathcal{F}$, the component $\mu_\mathcal{F} : \mathcal{F} \rightarrow \Gamma(-, p_\mathcal{F})$ of $\mu$ is defined as follows:
\[
\mu_{\mathcal{F}, U} : \mathcal{F}(U) \rightarrow \Gamma(U, p_\mathcal{F}); s \mapsto [x \mapsto s_x]
\]\par
The counit map ${\bf p} \circ {\bf \Gamma} \rightarrow Id$ of the adjunction ${\bf p} \dashv {\bf \Gamma}$ is isomorphic.
A component of the unit map $\mu_\mathcal{F} : \mathcal{F} \rightarrow \Gamma(-; p_\mathcal{F})$ is an isomorphism if and only if the presheaf $\mathcal{F}$ is a sheaf.
Therefore, the restruction of the adjunction ${\bf p} \dashv {\bf \Gamma}$ induces an equivalence ${\bf Sh}(X) \simeq \mbox{\bf \'{E}t}_X$.
	\section{\'{E}tale groupoids}
A {\it topological groupoid} (over $X$) is a groupoid object in the category of the topological spaces (such that the object space is $X$). 
A topological groupoid is {\it \'{e}tale} if the source map is a local homeomorphism
(cf. Definition \ref{etale}).
Any \'{e}tale groupoid has a sheaf associated with the \'{e}tale groupoid (cf. Section 2.2).
	\subsection{\'{E}tale groupoids}
Recall, a (small) {\it groupoid} $\mathcal{G}$ is a (small) category $\mathcal{G}$ such that any morphism in $\mathcal{G}$ is an isomorphism.
(In this paper, we assume that a groupoid is small unless it is confusing.)
A {\it topological groupoid} is a groupoid $\mathcal{G} = (\mathcal{G}_0,\mathcal{G}_1,s,t,i,inv,comp)$ such that the set of objects $\mathcal{G}_0$ and the set of arrows $\mathcal{G}_1$ are topological spaces, and that the structure maps (i.e. source map $s$, target map $t$, identities map $i$, inversion map $inv$ and composition map $comp$) are continuous.
\begin{defi} \label{etale}
A topological groupoid is {\it \'{e}tale} if the source map is a local homeomorphism.
\end{defi}
\begin{rem}
A topological groupoid is \'{e}tale if and only if the all structure maps are local homeomorphisms.
\end{rem}
In this paper, we fix the object space $\mathcal{G}_0$.
\begin{defi}
Let $X$ be a topological space.
A {\it topological groupoid over $X$} is a topological groupoid such that the object space is $X$.
\end{defi}
Morphisms between topological groupoids are continuous functors.
\begin{defi}
Let $\mathcal{G}$ and $\mathcal{H}$ be topological groupoids over a topological space $X$.
A {\it morphism $\phi : \mathcal{G} \rightarrow \mathcal{H}$ over $X$} is a functor $\phi : \mathcal{G} \rightarrow \mathcal{H}$ such that the object map of $\phi$ is the identity map of $X$ and that the arrow map of $\phi$ is continuous.
\end{defi}
	\subsection{Sheaves associated with \'{e}tale groupoids}
Let $\mathcal{G}$ be an \'{e}tale groupoid over $X$.
We define $\widehat{\mathcal{G}} (U,V)$ as the set $\{ f : U \rightarrow t^{-1} (V) | s \circ f = id \}$ for each open sets $U,V \in X_{top}$. 
The correction of these sets $\widehat{\mathcal{G}} (U,V)$ define a small category $\widehat{\mathcal{G}}$ such that the objects are the open sets in $X$.
(The composition $(-) \bullet (-)$ of the category $\widehat{\mathcal{G}}$ is defined as $(g \bullet f)(x) = g(t(f(x))) \cdot f(x)$ for $x \in U$, where the operation $(-) \cdot (-)$ is the composition of the groupoid $\mathcal{G}$.)
The category $\widehat{\mathcal{G}}$ includes the category $X_{top}$ as a subcategory.\par
For each open set $V \in X_{top}$ and each point $x \in X$, let $\widehat{\mathcal{G}}_x (V) = \displaystyle{\lim_{\underset{U \ni x}{\longrightarrow}}} \widehat{\mathcal{G}} (U,V)$. 
For each points $x,y \in X$, let $\widehat{\mathcal{G}}_x^y = \displaystyle{\lim_{\underset{V \ni y}{\longleftarrow}}} \widehat{\mathcal{G}}_x (V)$.
\begin{prop} \label{groupoidtosheaf}
The above $\widehat{\mathcal{G}}$ satisfies the following:
\begin{enumerate}
\renewcommand{\labelenumi}{(\arabic{enumi})}
  \item $Ob(X_{top}) = Ob(\widehat{\mathcal{G}})$.
  \item[(2.1)] The natural projections $\widehat{\mathcal{G}}_x^y \rightarrow \widehat{\mathcal{G}}_x (V)$ are injections. \\
(This identifies the set $\widehat{\mathcal{G}}_x^y$ as the set $\operatorname{Im}[\widehat{\mathcal{G}}_x^y \rightarrow \widehat{\mathcal{G}}_x (V)]$.)
  \item[(2.2)] The map $\displaystyle{\coprod_{y \in V}} \widehat{\mathcal{G}}_x^y \rightarrow \widehat{\mathcal{G}}_x (V)$ is surjection.
  \item[(2.3)] For three points $x, y, z \in X$, if we have a relation $y \rightarrow z$ (cf. Definition \ref{top}), then we obtain the inclusion $\widehat{\mathcal{G}}_x^y \subset \widehat{\mathcal{G}}_x^z$.
  \item[(2.4)] For a germ $f_x \in \widehat{\mathcal{G}}_x (V)$, $f_x$ belongs to $\widehat{\mathcal{G}}_x^y$ if and only if we have the relation $f(x) \rightarrow y$.
\setcounter{enumi}{2}
  \item The presheaf $\widehat{\mathcal{G}} (-,V)$ is a sheaf for each $V \in X_{top}$.
\end{enumerate}
\end{prop}
\begin{proof}
The claims {\it (1)} and {\it (3)} is obvious by definition. \par
We will prove the claim {\it (2.1)}.
The map $\widehat{\mathcal{G}} (U,V_1) \rightarrow \widehat{\mathcal{G}} (U,V_2)$ is an injection by definition for any open sets $U$, $V_1$ and $V_2$ ($V_1 \subset V_2$).
The map $\widehat{\mathcal{G}}_x (V_1) \rightarrow \widehat{\mathcal{G}}_x (V_2)$ is an injection for any point $x$ and any open sets $V_1$ and $V_2$ ($V_1 \subset V_2$).
Therefore, the set $\widehat{\mathcal{G}}_x^y (= \displaystyle{\lim_{\underset{V \ni y}{\longleftarrow}}} \widehat{\mathcal{G}}_x (V))$ is identified with the set $\displaystyle{\bigcap_{V \ni y}} \widehat{\mathcal{G}}_x (V)$.
In other words, the natural projections $\widehat{\mathcal{G}}_x^y \rightarrow \widehat{\mathcal{G}}_x (V)$ are injections.\par
We identify the set $\widehat{\mathcal{G}}_x^y$ as the set $\operatorname{Im}[\widehat{\mathcal{G}}_x^y \rightarrow \widehat{\mathcal{G}}_x (V)]$ by the claim {\it (2.1)}.
For a germ $f_x \in \widehat{\mathcal{G}}_x (V)$, $f_x$ belongs to $\widehat{\mathcal{G}}_x^y$ if and only if, for any open neighborhood $V'$ of $y$, there exist an open neighborhood $U$ of $x$ and a section $f \in \widehat{\mathcal{G}}(U, V')$ such that the germ of $f$ at $x$ coincides with $f_x$.
We obtain the claim {\it (2.2)} because $f_x$ belongs to the set $\widehat{\mathcal{G}}_x^{f(x)}$ for any $f_x \in \widehat{\mathcal{G}}_x (V)$.\par
We will prove the claim {\it (2.3)}.
Suppose that we have the relation $y \rightarrow z$.
Take any germ $f_x \in \widehat{\mathcal{G}}_x^y$.
For any open neighborhood $V$ of $y$, there exist an open neighborhood $U$ of $x$ and a section $f \in \widehat{\mathcal{G}}(U, V)$ such that the germ of $f$ at $x$ coincides with $f_x$.
Any open neighborhood of $z$ is an open neighborhood of $y$ because of the relation $y \rightarrow z$.
Therefore, the germ $f_x$ belongs to the set $\widehat{\mathcal{G}}_x^z$.\par
We will prove the claim {\it (2.4)}.
A proof of the implication $f(x) \rightarrow y \Rightarrow f_x \in \widehat{\mathcal{G}}_x^y$ is the same as the proof of the claim {\it (2.3)}.
We will prove the inverse.
Suppose that a germ $f_x \in \widehat{\mathcal{G}}_x(V)$ belongs to $\widehat{\mathcal{G}}_x^y$.
Take any open neighborhood $V'$ of $y$.
There exist an open neighborhood $U$ of $x$ and a section $f \in \widehat{\mathcal{G}}(U, V')$ such that the germ of $f$ at $x$ coincides with $f_x$.
The point $f(x)$ belongs to the open set $V'$.
Then, $V'$ is an open neighborhood of $f(x)$.
Therefore, we obtain the relation $f(x) \rightarrow y$.
\end{proof}
\begin{rem}
Suppose that $X$ is a $T_1$ space (cf. Definition \ref{top}).
If there exists a germ $f_x$ belonging to $\widehat{\mathcal{G}}_x^y \cap \widehat{\mathcal{G}}_x^z$, then we obtain the equality $y = f(x) = z$ because of the relation $y \leftarrow f(x) \rightarrow z$.
Therefore, if $X$ is a $T_1$ space, all together the above conditions {\it (2.1)}, {\it (2.2)}, {\it (2.3)} and {\it (2.4)} equivalents the following:
\begin{enumerate}
\renewcommand{\labelenumi}{(\arabic{enumi})}
\setcounter{enumi}{1}
  \item The cannonical map $\displaystyle{\coprod_{y \in V}} \widehat{\mathcal{G}}_x^y \rightarrow \widehat{\mathcal{G}}_x (V)$ induced by the natural projections $\widehat{\mathcal{G}}_x^y \rightarrow \widehat{\mathcal{G}}_x (V)$ is an isomorphism for each open set $V \in X_{top}$ and each point $x \in X$.
\end{enumerate}
\end{rem}
	\section{Pseudogroup sheaves} \label{pseudogroup}
A {\it pseudogroup} is a subgroupoid $\mathcal{H}$ of $\operatorname{Homeo}_X$ satisfying the sheaf property.
We define a {\it pseudogroup sheaf} (cf. Definition \ref{pseudogroupsheaf}), a new generalization of a pseudogroup.
The category of the pseudogroup sheaves on $X$ is equivalent to the category of the \'{e}tale groupoids over $X$ (cf. Section 4).
The interpretation of a pseudogroup in classical sence as a pseudogroup sheaf is given in Section 5.
	\subsection{Pseudogroup sheaves}
Let $X$ be a topological space. 
Suppose that $\mathcal{C}$ is a small category with including $X_{top}$ as a subcategory, and $Ob(X_{top}) = Ob(\mathcal{C})$. 
Then, the functor $\mathcal{C} (-,V) : X_{top}^{op} \subset \mathcal{C}^{op} \rightarrow {\bf Set}$ is a presheaf on $X$ for each $V \in X_{top}$. \par
For each open set $V \in X_{top}$ and each point $x \in X$, let $\mathcal{C}_x (V) = \displaystyle{\lim_{\underset{U \ni x}{\longrightarrow}}} \mathcal{C} (-,V)$. 
For each points $x,y \in X$, let $\mathcal{C}_x^y = \displaystyle{\lim_{\underset{V \ni y}{\longleftarrow}}} \mathcal{C}_x (V)$. \par
Suppose that the cannonical map $\displaystyle{\coprod_{y \in V}} \mathcal{C}_x^y \rightarrow \mathcal{C}_x (V)$ induced by the natural projections $\mathcal{C}_x^y \rightarrow \mathcal{C}_x (V)$ is an isomorphism for each open set $V \in X_{top}$ and each point $x \in X$.
We identify the set $\widehat{\mathcal{G}}_x^y$ as the set $\operatorname{Im}[\widehat{\mathcal{G}}_x^y \rightarrow \widehat{\mathcal{G}}_x (V)]$.
Then, the correction of these sets $\mathcal{C}_x^y$ define a small category $\widehat{\mathcal{G}}$ such that the sets of the objects is $X$ and that the sets of the arrows from $x$ to $y$ is $\mathcal{C}_x^y$.
(The composition $(-) \cdot (-)$ of the category $\mathcal{C}^\star$ is defined as $(g_y \cdot f_x) = (g \bullet f)_x$ for $x \in U$, where the operation $(-) \bullet (-)$ is the composition of the ctegory $\mathcal{C}$.)
\begin{lem}
The above composition $(-) \cdot (-)$ of the category $\mathcal{C}^\star$ is well-defined.
\end{lem}
\begin{proof}
Take a composable pair $(g_y, f_x) \in \mathcal{C}_y^z \times \mathcal{C}_x^y$.
For any open neighborhood $W$ of $z$, there exist an open neighborhood $V$ of $y$ and a section $g \in \mathcal{C}(V, W)$ such that the germ of $g$ at $y$ coincides with $g_y$.
There exist an open neighborhood $U$ of $x$ and a section $f \in \mathcal{C}(U, V)$ such that the germ of $f$ at $x$ coincides with $f_x$.
The composition $g_y \cdot f_x$ is defined as the germ $(g \bullet f)_x$. \par
Take any sections $g' \in \mathcal{C}(V', W)$ and $f' \in \mathcal{C}(U', V')$ such that the germ of $g'$ at $y$ coincides with $g_y$ and that the germ of $f'$ at $x$ coincides with $f_x$.
There exist an open neighborhood $V'' (\subset V \cap V')$ of $y$ and a section $g'' \in \mathcal{C}(V'', W)$ such that the restrictions $g |_{V''}$ and $g' |_{V''}$ coinside with $g''$.
There exist an open neighborhood $U''$ of $x$ and a section $f'' \in \mathcal{C}(U'', V'')$ such that the restrictions $f |_{U'}$ and $f' |_{U''}$ coinside with $f''$.
Then, we obtain the equality $(g \bullet f)_x = (g'' \bullet f'')_x = (g' \bullet f')_x$.
Therefore, the operation $(-) \cdot (-)$ is well-defined.
\end{proof}
\begin{defi} \label{pseudogroupsheaf}
Let $X$ be a $T_1$ space. 
A {\it pre-pseudogroup} on $X$ is a small category $\mathcal{C}$ with an embedding functor $X_{top} \subset \mathcal{C}$ satisfying the following:
\begin{enumerate}
\renewcommand{\labelenumi}{(\arabic{enumi})}
  \item $Ob(X_{top}) = Ob(\mathcal{C})$.
  \item The cannonical map $\displaystyle{\coprod_{y \in V}} \mathcal{C}_x^y \rightarrow \mathcal{C}_x (V)$ induced by the natural projections $\mathcal{C}_x^y \rightarrow \mathcal{C}_x (V)$ is an isomorphism for each open set $V \in X_{top}$ and each point $x \in X$.
  \item the category $\mathcal{C}^\star$ is a groupoid.
\end{enumerate}\par
A {\it pseudogroup sheaf} on $X$ is a pre-pseudogroup on $X$ satisfying the following:
\begin{enumerate}
\renewcommand{\labelenumi}{(\arabic{enumi})}
\setcounter{enumi}{3}
  \item the presheaf $\mathcal{C} (-,V)$ is a sheaf for any $V \in X_{top}$.
\end{enumerate}
\end{defi}
\begin{rem}
When $X$ is not a $T_1$ space, this definition is possible, but could not be reasonable (cf. Proposition \ref{groupoidtosheaf}).
\end{rem}
The first example is the psudogroup sheaf associated with an \'{e}tale groupoid.
\begin{ex}
Let $\mathcal{G}$ be a \'{e}tale groupoid, and let $\widehat{\mathcal{G}}$ be the category defined in Section 2.2.
If $X$ is a $T_1$ space, then the category $\widehat{\mathcal{G}}$ satisfies the conditions $(1)$, $(2)$ and $(4)$ in Definition \ref{pseudogroupsheaf}. 
In fact, the category $\widehat{\mathcal{G}}$ is a pseudogroup sheaf on $X$ because of Proposition \ref{groupoidiso}.
\end{ex}
\begin{prop} \label{groupoidiso}
The groupoid $\mathcal{G}$ is isomorphic to the category $\widehat{\mathcal{G}}^\star$.
In particular, the ctegory $\widehat{\mathcal{G}}^\star$ is a groupoid.
\end{prop}
\begin{proof}
Define a functor $\phi : \widehat{\mathcal{G}}^\star \rightarrow \mathcal{G}$ as the follows:
\[
\phi(f_x) = f(x).
\]
In fact, the map $\phi$ is compatible with the operations by the follows:
\[
\begin{array}{lcll}
\phi(g_y \cdot f_x)
& = & \phi(g \bullet f)_x) & \\
& = & (g \bullet f)(x) & \\
& = & g(y) \cdot f(x) & \mbox{(cf. Section 2.2)} \\
& = & \phi(g_y) \cdot \phi(f_x). &
\end{array}
\]
(Remark that the pair $(g_y, f_x)$ is a composable pair, therefore, we obtain the equality $y = t(f(x))$.)
The functor $\phi$ is a bijection because the source map of the groupoid $\mathcal{G}$ is a local homeomorphism.
Therefore, we obtain an isomorphism $\phi : \widehat{\mathcal{G}}^\star \rightarrow \mathcal{G}$.
\end{proof}
We give other examples.
\begin{ex}
The category $X_{top}$ is a pseudogroup sheaf on $X$.
\end{ex}
\begin{ex} \label{localhomeo}
Let $\operatorname{LoHomeo}_X (U,V) = \{ f : U \rightarrow V : \mbox{local homeomorphisms} \}$ for each open sets $U,V \in X_{top}$. 
The correction of these sets $\operatorname{LoHomeo}_X (U,V)$ define a small category $\operatorname{LoHomeo}_X$ such that the objects are the open sets in $X$.
The category $\operatorname{LoHomeo}_X$ satisfies the conditions $(1)$, $(3)$ and $(4)$ in Definition \ref{pseudogroupsheaf}. 
If $X$ is a $T_1$ space, then $\operatorname{LoHomeo}_X$ is a pseudogroup sheaf on $X$.
\end{ex}
\begin{ex}
Let $\mathcal{F}$ be a sheaf of group on $X$. 
We define $\mathcal{C}$ as 
\[
  \mathcal{C} (U,V) = \begin{cases}
    \mathcal{F} (U) & (U \subset V) \\
    \emptyset & (U \not\subset V).
  \end{cases}
\]
Then $\mathcal{C}$ is a pseudogroup sheaf on $X$.
\end{ex}
	\subsection{\'{E}tale groupoids associated with pre-pseudogroups}
Let $\mathcal{C}$ be a pre-pseudogroup on a $T_1$ space $X$. 
Let $s : E_{\mathcal{C}} \rightarrow X$ be an \'{e}tale space associated with the presheaf $\mathcal{C} (-,X)$. 
Then, $E_{\mathcal{C}} = \displaystyle{\coprod_{x \in X}} \mathcal{C}_x (X) \cong \displaystyle{\coprod_{x,y \in X}} \mathcal{C}_x^y$ and $s(f_x) = x$ for $f_x \in \mathcal{C}_x^y$
(cf. Section 1.4). 
In other words, $E_{\mathcal{C}}$ is the set of arrows of the groupoid $\mathcal{C}^\star$, and $s$ is the source map of $\mathcal{C}^\star$.
\begin{prop}
$\mathcal{C}^\star$ is an \'{e}tale groupoid.
\end{prop}
\begin{proof}
We have to proof that the identities map $i : X \rightarrow E_{\mathcal{C}}$, the inversion map $inv : E_{\mathcal{C}} \rightarrow E_{\mathcal{C}}$ and the composition map $comp : E_{\mathcal{C}} \times_X E_{\mathcal{C}} \rightarrow E_{\mathcal{C}}$ are continuous.
(Then, the target map $t : E_{\mathcal{C}} \rightarrow X$ is continuous because of the equality $t = s \circ inv$.)\par
For an open set $U \subset X$ and a section $f \in \mathcal{C}(U, X)$, a subset $[s, U] \subset E_\mathcal{C}$ is defined as $[s, U] = \{ f_x \, | \, x \in U \}$.
Recall that the collection of these $[f,U]$ generates the open sets of $E_{\mathcal{C}}$
(cf. Section 1.4).
\begin{description}
  \item[Claim.1] {\it $i$ is continuous.}
\end{description} \par
The map $i$ coincides with the inverse of the homeomorphism $s |_{[id_X, X]} : [id_X, X] \rightarrow X$.
Therefore, $i$ is continuous.
\begin{description}
  \item[Claim.2] {\it $inv$ is continuous.}
\end{description} \par
Take any point $g_y \in inv^{-1} ([f, U])$.
Denote $inv(g_y) = f_x$. 
This satisfies the equality $f_x \cdot g_y = id_y$.
$U$ is an open neighborhood of $x (= \overline{g}(y))$.
There exist an open neighborhood $V$ of $y$ and a section $g \in \mathcal{C}(V, U)$ such that the germ of $g$ at $y$ coincides with $g_x$.
We obtain the equality $(f \bullet g)_y = id_y$. 
There exists an open neighborhood $V' \subset V$ of $y$ such that $f \bullet g |_{V'} = id_{V'}$. 
We obtain the inclusion $g_y \in [g, V'] \subset inv^{-1} ([f, U])$.
Therefore, $inv$ is continuous.
\begin{description}
  \item[Claim.3] {\it $comp$ is continuous.}
\end{description} \par
Take any point $(g_y,f_x) \in comp^{-1} ([h, U])$. 
Denote $g_y \cdot f_x = h_x$. 
There exist an open neighborhood $V$ of $y$ and a section $g \in \mathcal{C}(V, X)$ such that the germ of $g$ at $y$ coincides with $g_y$.
There exist an open neighborhood $U' \subset U$ of $x$ and a section $f \in \mathcal{C}(U', V)$ such that the germ of $f$ at $x$ coincides with $f_x$.
Because of the equality $g_y \cdot f_x = h_x$, there exists an open neighborhood $U'' \subset U'$ of $x$ such that $g \bullet f |_{U''} = h_{U''}$. 
We obtain the inclusion $(g_y, f_x) \in [g, V] \times_X [f, U'] \subset comp^{-1} ([h, U])$.
Therefore, $comp$ is continuous.
\end{proof}
	\section{Main result}
Let $X$ be a $T_1$ space.
Some concrete categories are defined as follows.
\begin{quote}
{\bf PPse}$(X)$ the category of the pre-pseudogroups on $X$. \\
{\bf Pse}$(X)$ the category of the pseudogroup sheaves on $X$. \\
{\bf \'{E}tGrpd}$_X$ the category of the \'{e}tale groupoids over $X$.
\end{quote}
The construction $\mathcal{G} \mapsto \widehat{\mathcal{G}}$ (cf. Section 2.2) defines a functor $G : \mbox{\bf \'{E}tGrpd}_X \rightarrow {\bf Pse}(X) \subset {\bf PPse}(X)$.
The construction $\mathcal{C} \mapsto \mathcal{C}^\star$ (cf. Section 3.2) defines a functor $F : {\bf PPse}(X) \rightarrow \mbox{\bf \'{E}tGrpd}_X$.
We will prove that there exists an adjunction $F \dashv G : {\bf PPse}(X) \rightarrow \mbox{\bf \'{E}tGrpd}_X$. \par
We will construct the unit map $\mu : Id \rightarrow GF$ of the adjunction $F \dashv G$.
Let $\mathcal{C}$ be a pre-pseudogroup on $X$.
Define a morphism $\mu_\mathcal{C} : \mathcal{C} \rightarrow \widehat{\mathcal{C}^\star}$ over $X$ as follows:
\[
\mu_\mathcal{C}(f)(x) = f_x \, (\mbox{for $f \in  \mathcal{C}(U, V)$ and $x \in U$}).
\]
Remark that the morphism $\mathcal{C}(-, V) \rightarrow \widehat{\mathcal{C}^\star}(-, V)$ between presheaves induced by $\mu_\mathcal{C}$ is the sheafification
(cf. Section 1.5).
\begin{prop}
Let $\mathcal{C}$ and $\mu_\mathcal{C}$ be as above.
For any \'{e}tale groupoid $\mathcal{G}$ on $X$ and any morphism $\phi : \mathcal{C} \rightarrow \widehat{\mathcal{G}}$ between pre-pseudogroups, there exists a unique morphism $\bar{\phi} : \mathcal{C}^\star \rightarrow \mathcal{G}$ between \'{e}tale groupoids such that the following diagram is commutative:
	\[\xymatrix{
		\mathcal{C} \ar[r]^-{\mu_\mathcal{C}} \ar[dr]_-{\phi}
		& \widehat{\mathcal{C}}^\star \ar[d]^-{G(\bar{\phi})}
	\\
		& \widehat{\mathcal{G}}
	}\]
\end{prop}
\begin{proof}
The morphism $\mathcal{C}(-, X) \rightarrow \widehat{\mathcal{C}^\star}(-, X)$ between presheaves induced by $\mu_\mathcal{C}$ is the sheafification.
$s_{\mathcal{C}^\star} : \mathcal{C}^\star_1 \rightarrow X$ is the \'{e}tale space associated with the presheaf $\mathcal{C}(-, X)$.
The sheaf $\widehat{\mathcal{C}^\star}(-, X)$ (resp. $\widehat{\mathcal{G}}(-, X)$) is the sheaf associated with the \'{e}tale space $s_{\mathcal{C}^\star}$ (resp. $s_\mathcal{G} : \mathcal{G}_1 \rightarrow X$, where $s_\mathcal{G}$ is the source map of the groupoid $\mathcal{G}$).
The morphism $\phi : \mathcal{C}(-, X) \rightarrow \widehat{\mathcal{G}}(-, X)$ between presheaves induces a unique morphism $\bar{\phi}_1 : s_{\mathcal{C}^\star} \rightarrow s_{\mathcal{G}}$ such that the following diagram is commutative:
	\[\xymatrix{
		\mathcal{C}(-, X) \ar[r]^-{\mu_\mathcal{C}} \ar[dr]_-{\phi}
		& \widehat{\mathcal{C}}^\star(-, X) \ar[d]^-{{\bf \Gamma}(\bar{\phi}_1)}
	\\
		& \widehat{\mathcal{G}}(-, X),
	}\]
(where the functor ${\bf \Gamma}$ is the same as in Section 1.5.)
Then, the map $\bar{\phi}_1 : \mathcal{C}^\star_1 \rightarrow \mathcal{G}_1$ can be written concretely as follows:
\[
\bar{\phi}_1(f_x) = \phi(f)(x).
\]
Write the target maps as $t$, and the following equality holds:
\[
\begin{array}{lcl}
t(\bar{\phi}_1(f_x))
& = & t(\phi(f)(x)) \\
& = & t(\phi(f)_x) \\
& = & t(f_x).
\end{array}
\]
For any composable pair $(g_y, f_x)$, the following equality holds:
\[
\begin{array}{lcl}
\bar{\phi}_1(g_y \cdot f_x)
& = & \bar{\phi}_1((g \bullet f)_x) \\
& = & \phi(g \bullet f)(x) \\
& = & (\phi(g) \bullet \phi(f))(x) \\
& = & \phi(g)(y) \cdot \phi(f)(x) \\
& = & \bar{\phi}_1(g_y) \cdot \bar{\phi}_1(f_x).
\end{array}
\]
Therefore, the map $\bar{\phi}_1$ determines a morphism $\bar{\phi} : \mathcal{C}^\star \rightarrow \mathcal{G}$ between \'{e}tale groupoids.
Then, the following diagram is commutative:
	\[\xymatrix{
		\mathcal{C} \ar[r]^-{\mu_\mathcal{C}} \ar[dr]_-{\phi}
		& \widehat{\mathcal{C}}^\star \ar[d]^-{G(\bar{\phi})}
	\\
		& \widehat{\mathcal{G}}
	}\]
Conversely, suppose there exists a morphism $\psi : \mathcal{C}^\star \rightarrow \mathcal{G}$ that makes the above diagram commutative.
Then, the following diagram is commutative:
	\[\xymatrix{
		\mathcal{C}(-, X) \ar[r]^-{\mu_\mathcal{C}} \ar[dr]_-{\phi}
		& \widehat{\mathcal{C}}^\star(-, X) \ar[d]^-{{\bf \Gamma}(\psi_1)}
	\\
		& \widehat{\mathcal{G}}(-, X)
	}\]
We obtain the equality $\psi_1 = \bar{\phi}_1$ because of the universality of the sheafification.
This implies the equality $\psi = \bar{\phi}$.
Therefore, the morphism $\bar{\phi}$ is unique.
\end{proof}
\begin{prop}
The counit map $\epsilon : FG \rightarrow Id$ of the adjunction $F \dashv G$ is isomorphic.
\end{prop}
\begin{proof}
For a \'{e}tale groupoid $\mathcal{G}$, the morphism $\epsilon_\mathcal{G} : \widehat{\mathcal{G}}^\star \rightarrow \mathcal{G}$ is the unique morphism such that the following diagram is commutative:
	\[\xymatrix{
		\widehat{\mathcal{G}} \ar[r]^-{\mu_{\widehat{\mathcal{G}}}} \ar[dr]_-{id}
		& \widehat{\widehat{\mathcal{G}}^\star} \ar[d]^-{G(\epsilon_\mathcal{G})}
	\\
		& \widehat{\mathcal{G}}.
	}\]
Then, the following diagram is commutative:
	\[\xymatrix{
		\widehat{\mathcal{G}}(-, V) \ar[r]^-{\mu_{\widehat{\mathcal{G}}}} \ar[dr]_-{id}
		& \widehat{\widehat{\mathcal{G}}^\star}(-, V) \ar[d]^-{{\bf \Gamma}(\epsilon')}
	\\
		& \widehat{\mathcal{G}}(-, V),
	}\]
where $\epsilon'$ is the morphism between \'{e}tale spaces determined by the arrow map $(\epsilon_\mathcal{G})_1 : \widehat{\mathcal{G}}^\star_1 \rightarrow \mathcal{G}_1$.
These morphisms $\epsilon'$ coincide with the components of the counit map of the adjunction ${\bf p} \dashv {\bf \Gamma} : {\bf PSh}(X) \rightarrow \mbox{\bf \'{E}t}_X$ (cf. Section 1.5) because of the universality of the sheafification.
These are isomorphisms.
Therefore, $\epsilon$ is isomorphic.
\end{proof}
A component of the unit map $\mu_\mathcal{C} : \mathcal{C} \rightarrow \widehat{\mathcal{C}^\star}$ is an isomorphism if and only if the pre-pseudogroup $\mathcal{C}$ is a pseudogroup sheaf.
Therefore, the restriction of the adjunction $F \dashv G$ induces an equivalence ${\bf Pse}(X) \simeq \mbox{\bf \'{E}tGrpd}_X$.
We obtain the following theorem.
\begin{thm} \label{main}
Let $X$ be a $T_1$ space. 
Then, there exists an equivalence ${\bf Pse}(X) \simeq \mbox{\bf \'{E}tGrpd}_X$.
\end{thm}
	\section{Classical pseudogroups} 
Let $X$ be a $T_1$ space, and let $\mathcal{C}$ be a pre-pseudogroup on $X$. 
Any section $f \in \mathcal{C} (U,V)$ has the {\it underlying map} $\overline{f} : U \rightarrow V$ in the following way: 
Take any point $x \in U$. 
The germ $f_x$ of $f$ at $x$ belongs to $\mathcal{C}_x (V) \cong \displaystyle{\coprod_{y \in V}} \mathcal{C}_x^y$. 
Then, there exists exactly one point $y \in V$ such that the germ $f_x$belongs to $\mathcal{C}_x^y$. 
We define $\overline{f} (x) = y$.\par
The map $\overline{f} : U \rightarrow V$ is a local homeomorphism because of the equality $\overline{f} = t \circ (s |_{[f, U]})^{-1}$, where $t$ is the target map and $s$ is the source map.
Moreover, the correspondence $f \mapsto \overline{f}$ define a morphism over $X$ from the pre-pseudogroup $\mathcal{C}$ to the pseudogroup sheaf $\operatorname{LoHomeo}_X$
(cf. Example \ref{localhomeo}). 
\begin{lem}
\mbox{}
\begin{itemize}
\item $\overline{id} = id$.
\item $\overline{g \bullet f} = \overline{g} \circ \overline{f}$.
\end{itemize}
\end{lem}
\begin{proof}
We will prove the equality $\overline{id} = id$.
Take any identity $id \in \mathcal{C}(U, U)$ and any point $x \in U$.
For any open neighborhood $V \subset U$ of $x$, the restriction $id |_V$ of $id$ belongs to the set $\mathcal{C}(V, V)$.
Therefore, the germ $id_x$ belongs to the set $\mathcal{C}_x^x$.
We obtain the equality $\overline{id}(x) = x$. \par
We will prove the equality $\overline{g \bullet f} = \overline{g} \circ \overline{f}$.
Take any composable pair $(g, f) \in \mathcal{C}(V, W) \times \mathcal{C}(U, V)$ and any point $x \in U$.
Let $y = \overline{f}(x)$.
The germ $g_y$ beongs to the set $\mathcal{C}_y^{\overline{g}(y)}$.
For any open neighborhood $W' \subset W$ of $\overline{g}(y)$, there exist an open neighborhood $V' \subset V$ of $y$ and a section $g \in \mathcal{C}(V', W')$ such that the germ of $g$ at $y$ coincides with $g_y$.
There exist an open neighborhood $U' \subset U$ of $x$ and a section $f \in \mathcal{C}(U', V')$ such that the germ of $f$ at $x$ coincides with $f_x$.
Therefore, the germ $(g \bullet f)_x$ belongs to the set $\mathcal{C}_x^{\overline{g}(y)}$.
We obtain the equality $(\overline{g \bullet f})(x) = \overline{g}(y) = (\overline{g} \circ \overline{f})(x)$.
\end{proof}
A pseudogroup in classical sence corresponds with a {\it concrete} pseudogroup sheaf.
\begin{defi}
A pseudogroup sheaf $\mathcal{C}$ on $X$ is {\it concrete} if the functor $\mathcal{C} \rightarrow \operatorname{LoHomeo}_X;f \mapsto \overline{f}$ is faithful.
\end{defi}
Let $\mathcal{C}$ be a concrete pseudogroup on $X$. 
Then the subcategory of invertible morphisms of $\mathcal{C}$ is a {\it pseudogroup} in a classical sence.\par
Conversely, let $\mathcal{Q}$ be a pseudogroup in a classical sence.
For two open sets, define a set $\mathcal{C'}(U, V)$ by the follows:
\[
\mathcal{C'}(U, V) = \{ i \circ f \, | \, f \in \mathcal{Q}(U, V'), \mbox{$i$ is the inclusion map $V' \hookrightarrow V$} \}.
\]
The correction of these sets $\mathcal{C'}(U, V)$ determines a pre-pseudogroup $\mathcal{C'}$.
We can obtain a concrete pseudogroup sheaf $\mathcal{C}$ by the sheafification of the pre-pseudogroup $\mathcal{C'}$.

\bibliography{groupoid,category,complex}
\bibliographystyle{plain}


\end{document}